\documentclass[12pt,oneside,reqno]{amsart}
\usepackage[all]{xy}
\usepackage{amsfonts,amsmath,oldgerm,amssymb,amscd}
\UseComputerModernTips
\numberwithin{equation}{section}
\usepackage[breaklinks]{hyperref}
\allowdisplaybreaks

\usepackage{enumerate}
\usepackage{caption} 
\newtheorem{theorem}{Theorem}[section]

\newtheorem{lemma}[subsection]{{\bf Lemma}}

\newtheorem{remark}[subsection]{Remark}

\newcommand{\al}{\alpha}
\newcommand{\be}{\beta}
\newcommand{\ga}{\gamma}
\newcommand{\la}{\lambda}

\newcommand{\Z}{\mbox{$\mathbb Z$}}
\newcommand{\Q}{\mbox{$\mathbb Q$}}

\begin{document}

\title[Sum of terms of recurrences in the solutions of Pell equations]{Sum of terms of recurrence sequences in the solution sets of generalized Pell equations} 

\author[P. K. Bhoi]{P. K.  Bhoi}
\address{Pritam Kumar Bhoi, Department of Mathematics, National Institute of Technology Rourkela, Odisha-769 008, India}
\email{pritam.bhoi@gmail.com}

\author[R. Padhy]{R. Padhy}
\address{Rudranarayan Padhy, Department of Mathematics, National Institute of Technology Calicut, 
Kozhikode-673 601, India.}
\email{padhyrudranarayan1996@gmail.com}

\author[S. S. Rout]{S. S. Rout}
\address{Sudhansu Sekhar Rout, Department of Mathematics, National Institute of Technology Calicut, 
Kozhikode-673 601, India.}
\email{sudhansu@nitc.ac.in; lbs.sudhansu@gmail.com}


\begin{abstract}
Let $(X_{k})_{k\geq 1}$ and $(Y_k)_{k\geq 1}$ be the sequence of $X$ and $Y$-coordinates of the positive integer solutions $(x, y)$ of the equation $x^2 - dy^2 = t$. In this paper, we completely describe those recurrence sequences such that the sums of two terms of recurrence sequences in the solution sets of generalized Pell equations are infinitely many. Further, we give an upper bound for the number of such terms when there are only finitely many of them. This work is motivated by the recent paper of Hajdu and Sebesty{\'e}n (Int. J. Number Theory 18 (2022), 1605-1612).
\end{abstract}

\maketitle
\pagenumbering{arabic}
\pagestyle{headings}

\section{Introduction}

Let $d, t$ be nonzero integers with $d > 1$ square-free, and consider the equation
\begin{equation}\label{gpelleqn}
x^2 - dy^2 = t
\end{equation}
in integers $x, y$. If $t = \pm 1, \pm4$, then \eqref{gpelleqn} is called a Pell equation, while in case of general $t$ it is a generalized Pell equation. Furthermore, setting $(x_1, y_1)$ for the smallest such solution of the Pell equation $x^2 - dy^2 =1$, all other solutions are of the form $(x_m, y_m)$ where
\[x_m+\sqrt{d}y_m = (x_1+\sqrt{d}y_1)^m\quad \mbox{for}\quad m \geq 1.\]

Recently several mathematicians have investigated the following type of problem related to solution sets (i.e., $X$ and $Y$-coordinates) of Pell equation solutions.  Assume that $\mathbb{U}:= (U_n)_{n\geq 0}$ is some interesting sequence of positive integers and $\{(x_m, y_m)\}_{m\geq 1}$ are sequence of solutions of Pell equation \eqref{gpelleqn}. What can one say about the number of solutions of the containment $x_m\in \mathbb{U}$ for a generic $d$? What about the number of solutions of the containment $y_m\in \mathbb{U}$? For most of the binary recurrent sequences(Fibonacci numbers \cite{ birlucatogbe2018, lucatogbe2018}, tribonacci numbers \cite{luca 2017}, rep-digits in some given integer base $b\geq 2$ \cite{dlt2016, fl2018, glz2020}), the equation $x_m\in \mathbb{U}$ has at most one positive integer solution $m$ for any given $d$ except for finitely many values of $d$. This is best possible which follows from the fact that if $u\in \mathbb{U}\setminus \{1\}$, then $(x, y) = (u, 1)$ is a solution to $x^2-dy^2=1$ for $d:= u^2-1$. Similarly, if $\mathbb{U}$ contains $1$ and infinitely many even numbers then there are infinitely many $d$ such that $y_m \in \mathbb{U}$ has two positive integer solutions in $m$. If $\mathbb{U}$ is the value set of a binary recurrent sequence then $y_m\in \mathbb{U}$ has at most two solutions $m$ provided $d$ exceeds some computable bound depending on $\mathbb{U}$ (see \cite{efgl2022, fl2020, fl3, lucazottor2023} ). For other related questions in this direction, one can refer (\cite{bravo2018, dd2020}).

On the other hand, determining perfect powers and sum of $S$-units in linear recurrence sequences is another interesting problem that has received a lot of attention. Erazo et al. \cite{egl2020} show that there are only a finite number of non-square  integers $d>1$ such that there are at least two different elements of the sequence $(X_m)_{m \geq 1}$, which is $x$-coordinate of \eqref{gpelleqn} with $t=\pm 1$, that can be represented as a linear combination of prime powers with fixed primes and coefficients. For other related result, refer (\cite{egl2020,  egl2021, hs} and the references therein. 

Fuchs and Heintze \cite{fh2021} proved results concerning values of recurrence sequences in the coordinates of solutions of general norm form equations. Hajdu and Sebesty{\'e}n \cite{hs2022} completely describe those recurrence sequences which have infinitely many terms in either of the $x, y$ coordinates of the solution sets of generalized Pell equations, i.e., of \eqref{gpelleqn} with arbitrary $t$. Furthermore, they obtained an upper bound for the number of solutions in the case where there are only finitely many solutions. 

In this paper, we study a problem related to sum of terms recurrence sequences which have infinitely many terms in the solution sets of generalized Pell equations. More precisely, we give an upper bound for the number of such terms when there are only finitely many of them. 

\section{Notation and Main Result}
Let $a_1,\dots, a_k \in \Z$ with $a_k\neq 0$ and $U_0,\dots,U_{k-1}$ are integers not all zero and let $(U_{n})_{n \geq 0}$ be a $k^{th}$ order linear recurrence sequence defined by
\begin{equation}\label{eq4}
U_{n} = a_1U_{n-1} + \dots +a_kU_{n-k}.
\end{equation}
Let $\alpha_1,\dots,\alpha_q$ be the distinct roots of the corresponding characteristic polynomial  
\begin{equation}\label{eq5}
f(x):= x^k - a_1x^{k-1}-\dots-a_k.
\end{equation} 
Then for $n\geq 0$, we have 
\begin{equation}\label{eq6}
	U_n=f_1(n)\alpha_1^n +\cdots + f_q(n)\alpha_{q}^n,
\end{equation}
where $f_i(n)$ are nonzero polynomials with degrees less than the multiplicity of $\alpha_i$; the coefficients of $f_i(n)$ are elements of the field $\Q(a_1,\dots, a_k, U_0,\dots,U_{k-1}, \alpha_1, \ldots, \alpha_q)$. 

The sequence $(U_n)_{n \geq 0}$ is called {\it degenerate} if there are integers $i, j$ with $1\leq i< j\leq q$ such that $\alpha_i/\alpha_j$ is a root of unity; otherwise it is called {\it non-degenerate}. The sequence $(U_n)_{n \geq 0}$ is called {\it simple} if $q=k$. In this case, \eqref{eq6} becomes
\begin{equation}\label{eq6a}
	U_n=\eta_1\alpha_1^n +\cdots + \eta_k\alpha_{k}^n,
\end{equation}
where $\eta_i$'s are constants. We write $X$ and $Y$ for the sets of solutions of \eqref{gpelleqn} in $x\in \Z$ and $y\in \Z$, respectively. We prove the following theorem.
 \begin{theorem}\label{th1}
Let $(U_n)_{n \geq 0}$ be a simple non-degenerate linear recurrence sequence
of integers of order $k$ as in \eqref{eq6a}, such that its characteristic polynomial is not of the form $x^2 + ax \pm 1$ with $(a^2 \mp 4)/d$ being a square in $\Q$. Further, assume that the roots of the characteristic polynomial $\alpha_i, (1\leq i \leq k)$ are pairwise multiplicatively independent and none of the root is a root of unity.  Then 
\begin{equation}\label{eqmain}
U_{n_1} + U_{n_2} \in X \cup Y
\end{equation}
holds only for finitely many indices $(n_1, n_2)$. Further, the number of such values $(n_1, n_2)$ is bounded by $c_1 = c_1(k, d, t)$, where $c_1$ is an effectively computable constant depending only on $k, d, t$.
\end{theorem}
We have given several remarks below where we have shown that there are infinitely many indices $(n_1, n_2)$ of \eqref{eqmain} if one of the assumptions in Theorem \ref{th1} is not satisfied.  

\begin{remark}
Consider the linear recurrence sequence of order three defined by the recurrence relation
\begin{equation}\label{eqrec101}
U_{n+3} = 3U_{n+2} -3U_{n+1} + 2U_{n},\quad n\geq 0
\end{equation}
with initial values $U_0 = 0, U_1 = 1$ and $U_2 = 1$. The characteristic polynomial of $U_n$ is $x^3 -3x^2 +3x-2$ and its roots are $2, (1\pm \sqrt{3}i)/2$.  Take the Pell equation 
\begin{equation}\label{eqrem4}
x^2 - 2y^2 = 1.
\end{equation}
 Then there are infinitely many $(n_1, n_2)$ such that $U_{n_1}+U_{n_2}=2\in Y$, where $Y$ is  solution set of \eqref{eqrem4} in $y\in \Z$. Note that here the given recurrence in \eqref{eqrec101} is degenerate.
\end{remark}

\begin{remark}
Take $d = 3$ and consider the classical Pell equation
\begin{equation}\label{eqrem1}
x^2 - 3y^2 = 1.
\end{equation}
As it is well-known, its positive solutions are given by
\[x + \sqrt{3}y = (2 + \sqrt{3})^m \quad (m \geq 1).\]
Let $U_0 = 0$, $U_1 = 1$ and
\begin{equation}\label{rec}
U_{n+2} = 4U_{n+1} - U_n \quad (n \geq 0).
\end{equation}
There are infinitely many $(n_1, 0)$ such that $U_{n_1}+U_{n_2} = U_{n_1} \in Y$, where $Y$ is  the solution set of \eqref{eqrem1} in $y\in \Z$. On the other hand, here $a = -4$, and $(a^2 - 4)/3 = 4$ is a full square, yielding that the roots of the characteristic polynomials $x^2 - 4x + 1$ are units of the ring of integers of $\Q(\sqrt{3})$.
\end{remark}

\begin{remark}
Consider the Pell equation
\begin{equation}\label{eqrem3}
x^2 - 5y^2 = 1.
\end{equation}
Its positive solutions are given by
\[x + \sqrt{5}y = (9 + 4\sqrt{5})^m \quad (m \geq 0).\]
Let 
\begin{equation}\label{rec2}
U_n = 1^n-(-1)^n \quad (n \geq 0).
\end{equation}
There are infinitely many $(n_1, n_2)$ such that $U_{n_1}+U_{n_2}\in Y$, where $Y$ is  the solution set of \eqref{eqrem3} in $y\in \Z$. Note that both the roots are root of unity.
\end{remark}

Following type of example is not included in \cite{hs2022}.
\begin{remark}
Consider the linear recurrence sequence $U_n = 4U_{n-1} - 3 U_{n-2}$ with $U_0 =2$ and $U_1 = 2$. The characteristic polynomial of $U_n$ is $x^2-4x+3$ and the roots are $3$ and $1$. So, $U_n = 2$ for all $n\geq 0$. Take the Pell equation 
\begin{equation}\label{eqrem5}
x^2 - 5y^2 = -4.
\end{equation}
Then there are infinitely many $(n_1, n_2)$ such that $U_{n_1}+U_{n_2}= 4\in X$, where $X$ is the  solution set of \eqref{eqrem4} in $x\in \Z$. Here one of the roots of the characteristic polynomial of the recurrence sequence is a root of unity. 
\end{remark}

Our method of proof follows closely the general approach from \cite{hs2022}.

\section{Auxiliary results}

Let $x_1$ and $y_1$ be the smallest positive solutions (in $x$ and $y$, respectively) of the equation
\begin{equation}\label{pelleqn}
 x^2 - dy^2 = 1.
\end{equation}
It is well known that (see \cite[p.354]{niven1991}) all positive integer solutions $x, y$ of \eqref{pelleqn} are given by \[x + \sqrt{d}y =(x_1 + \sqrt{d}y_1)^m\] for some $m \geq 1$. The sequence $(x_k)_{k\geq 1}$ and $(y_k)_{k\geq 1}$ are in fact binary recurrence sequences. The following lemma shows that the sets of the coordinates of the
solutions of \eqref{gpelleqn} are unions of finitely many non-degenerate binary linear recurrence sequences.

\begin{lemma}\label{lem2}
 Let $x_1$ be the smallest positive solution of the equation \eqref{pelleqn}. If \eqref{gpelleqn} has a solution, then all its solutions are given by
\[(x, y) =(G^{(i)}_n ,H^{(i)}_n) \quad (i = 1, . . . , I)\]
with some binary recurrence sequences
\[G^{(i)} = G^{(i)}(2x_1,-1,G^{(i)}_0 ,G^{(i)}_1 ),\quad H^{(i)} = H^{(i)}(2x_1, -1,H^{(i)}_0 ,H^{(i)}_1 ).\]
Here $I$ and $G^{(i)}_0 ,G^{(i)}_1 ,H^{(i)}_0 ,H^{(i)}_1 $ $(i = 1, \ldots , I)$ are some positive integers
with $I < c_2$ and
\begin{equation}
|G^{(i)}_j|, |H^{(i)}_j| < c_3 \quad (0 \leq j \leq 1, 1 \leq i \leq I),
\end{equation}
where $c_2$ is an effectively computable constant depending only on $t$, while $c_3$ is an effectively computable constant depending only on $d$ and $t$.
\end{lemma}

\begin{proof}
See Lemma 3.2 in \cite{hs}.
\end{proof}

\begin{remark}\label{auxrem1}
By Lemma \ref{lem2}, we can say that $X \cup Y$ is the union of at most $2c_2$ binary recurrence sequences $(V_m)$ satisfying
\begin{equation}\label{vm}
V_{m+2} = 2x_1V_{m+1} - V_m \quad (m \geq 0).
\end{equation}
with some $V_0, V_1$ satisfying
\[|V_0|, |V_1| \leq c_3.\]
Here $c_2$ and $c_3$ are the constants appearing in Lemma \ref{lem2}. Since $x_1 > 1$, by \eqref{eq6} we get that
\[V_m=B\be^m+C\ga^m \quad (m \geq 0),\]
where $\be, \ga$ are the (real) roots of the polynomial $x^2 - 2x_1x + 1$ and $B,C$ are nonzero conjugate elements of $\Q(\sqrt{d})$. In particular, $\be, \ga$ are units and conjugates in $\Q(\sqrt{d}).$
\end{remark}

We need a result  concerning the finiteness of the solutions of so-called polynomial-exponential equations (see \cite{ss2000}). For its formulation, we need to introduce some further notation.

Consider the equation
\begin{equation}\label{smith}
\sum_{l=1}^{r}P_l(\bf{x})\al_{\it{l}}^{\bf{x}} = 0
\end{equation}
in variables ${\bf{x}} = (x_1, . . . , x_s) \in \Z^s,$ where the $P_l$ are polynomials with coefficients in an algebraic number field $K$, and
\[\al^{\bf{x}}_l = \al^{x_1}_{l_1} \cdots \al^{x_s}_{l_s}\]
with given nonzero $\al_{l_1}, \ldots ,\al_{l_s} \in K
 \quad (l = 1, \ldots , r).$
Let $\pi$ be a partition of the set $\{1, \ldots , r\}$. The sets $\lambda\subset \{1, \ldots, r\}$ belonging to $\pi$ will be considered to be elements of $\pi$: we write $\lambda \in \pi$. Given $\pi$, the set of equations
\begin{equation}\label{ref}
\sum_{l\in \la}P_l(\bf{x})\al_{\it{l}}^{\bf{x}} = 0, \quad \la \in \pi
\end{equation}
yields a refinement of \eqref{smith}. When $\pi'$ is a refinement of $\pi$, then system of equations corresponding to the partition $\pi'$ implies system of equations corresponding to the partition $\pi$. Let $S(\pi)$ be the set of solutions of \eqref{ref} which are not solutions of \eqref{ref} with any proper refinement $\pi'$ of $\pi$. 

Set $i \overset{\pi}{\sim} j$ if $i$ and $j$ lie in the same subset $\la$ belonging to $\pi$. Let $G(\pi)$ be the
subgroup of $\Z^s$ consisting of $\bf{z}$ with
\[\al_i^{\bf{z}} = \al_j^{\bf{z}}\] 
for any $i, j$ with $i \overset{\pi}{\sim} j$.

\begin{lemma}\label{gp}
Using the above notation, if $G(\pi) = \{0\}$ then we have
\[S(\pi)<2^{35A^3}D^{6A^2}\]
with $D = \deg(K)$ and
\[ A = \max \left(s, \sum_{l\in \Lambda}\binom{s+\delta_l}{s}\right),\]
where $\delta_l$ is the total degree of the polynomial $P_l.$
\end{lemma}

\begin{proof}
The statement is Theorem 1 in \cite{ss2000}.
\end{proof}

\begin{lemma}\label{lem3.4}
Let $(U_n)_{n\geq 0}$ be a simple linear recurrence sequence of integers of order at least two. Further, assume that the roots of the characteristic polynomial $\alpha_i, (1\leq i \leq r)$ are pairwise multiplicatively independent. Then for $w= (w_k, \ldots, w_1)\in \Z^k, w_1, \ldots, w_k\neq 0$, there are finitely many tuples $(n_1, \ldots, n_k)\in \Z^k$ satisfying
  \begin{equation}\label{eq1lem5}
  w_k{U}_{n_k}+ \cdots+w_1{U}_{n_1} =0
  \end{equation} and \begin{equation}\label{neqeq155}
w_k\alpha_{i}^{n_k}+ w_{k-1}\alpha_i^{n_{k-1}}+\cdots + w_1\alpha_i^{n_1} \neq 0, \quad (1\leq i \leq r),
\end{equation}
 with
 \begin{equation}\label{eq2lem5}
 \sum_{h \in J}w_h{U}_{n_h} \neq 0
 \end{equation}
for each proper non-empty subset  $J$ of $\{1, \ldots, k\}$.
\end{lemma}
\begin{proof}
See Proposition 4.1 in \cite{rout2022}.
\end{proof}
Now, we can give the proof of Theorem \ref{th1}.

\section{Proof of Theorem \ref{th1}}
Let $(U_n)_{n \geq 0}$ be a simple linear recurrence sequence of order
$k$ as given in \eqref{eq6a}, satisfying the assumptions of the statement. From Remark \ref{auxrem1}  and \eqref{eqmain}, the inclusion $U_{n_1}+U_{n_2} \in X \cup Y$ is equivalent with
\begin{equation}\label{n1n2}
B\be^m + C\ga^m = \sum_{i=1}^{k} \eta_i\al_i^{n_1} + \sum_{j=1}^{k} \eta_j\al_j^{n_2}
\end{equation}
for some $m \geq 0$. Note that by our convention on the minimality
on $k$ in \eqref{eq4}, here none of $\eta_i$ and $\al_i$ is zero; further, the degrees of the $\eta_i$ are zero. We shall show that \eqref{n1n2} has only finitely many solutions in $(n_1,n_2)$, whose number is effectively bounded in terms of $k, d, t$. This, in view of that $B,C, \be, \ga$ are coming from a finite set of at most $2c_2$ elements, implies our theorem.

We shall handle equation \eqref{n1n2} by Lemma \ref{gp}. For this, first we rewrite \eqref{n1n2} as 
\begin{equation}\label{uniteq}
\sum_{i=1}^{k} \eta_i\al_i^{n_1} + \sum_{\hat{i}=1}^{\hat{k}} \eta_{\hat{i}}\al_{\hat{i}}^{n_2}-B\be^m - C\ga^m = 0.
\end{equation}
Here we note that $\alpha_i = \alpha_{\hat{i}}$ and  $\eta_i= \eta_{\hat{i}}$ for all $1\leq i\leq k$. The $2k+2$ summands on the left hand side of \eqref{uniteq} will be parametrised by the symbols
\[1, 2, \ldots, k, \hat{1}, \ldots, \hat{k}, k+\hat{k}+1,k+\hat{k}+2.\]
Suppose that the positive integers $n_1, n_2, m$ are solutions of \eqref{uniteq}. Write ${\bf N} = (n_1,n_2,m)$,
\[h_i({\bf N})= \begin{cases} \eta_i, \quad \mbox{if}\quad i=1, \ldots k \\
 \eta_{\hat{i}}, \quad \mbox{if}\quad i=\hat{1}, \ldots, \hat{k} \\
-B, \quad \mbox{if} \quad i= k+\hat{k}+1\\
-C, \quad \mbox{if}\quad  i=k+\hat{k}+2
\end{cases}
\quad \mbox{and} \quad {\bf \delta}_i= \begin{cases} (\al_i,1,1), \quad \mbox{if}\quad i=1, \ldots, k, \\
(1,\al_i,1),\quad \mbox{if} \quad i=\hat{1}, \ldots, \hat{k} \\
(1,1,\be), \quad \mbox{if} \quad i=k+\hat{k}+1\\
(1,1,\ga), \quad \mbox{if}\quad  i= k+\hat{k}+2.
\end{cases}\]
Let $\pi$ be a partition of the set $\{1,\ldots ,k, \hat{1}, \ldots, \hat{k}, k+\hat{k}+1,k+\hat{k}+2\}$ such that we have
\begin{equation}\label{eq10}
\sum_{i\in \la}h_i({\bf N}){\bf \delta}_i^{N} = 0, \quad \la \in \pi
\end{equation}
but \eqref{eq10} does not hold for any proper refinement of $\pi$. Observe that all solutions $(n_1,n_2,m)$ of \eqref{uniteq} are solutions of \eqref{eq10} with some $\pi$.

Assume first that there is a subset $\la$ of  $\pi$ such that $\la \subseteq \{1, \ldots , k\}$. If $\la$ is singleton set, that is, $\la = \{i\}$,  and $(\pi)$ yields $\eta_i\alpha_i^{n_1} =0$, which is not possible as $\al_i \neq 0, \; (1 \leq i \leq k)$. So, $|\la| =1$ is not possible. Thus, $|\la| \geq 2$. If  we had $i \overset{\pi}{\sim} j$ with some $i\neq j$ then this yields $\alpha_i^{n_1}  = -\alpha_j^{n_1}$. Since the recurrence sequence $(U_n)_{n \geq 0}$ is non-degenerate, we have $G(\pi) = \{(0, 0, 0)\}$. Hence, by Lemma \ref{gp} we get an upper bound for the number of these values of $n_1$ in terms of $k$.

Similarly, if we assume there is a subset $\la$ of  $\pi$ such that $\la \subseteq \{\hat{1}, \ldots, \hat{k}\}$, then by proceeding as above  we get an upper bound for the number of these values of $n_2$ in terms of $k$.

Now, assume that $\la \subseteq \{1, \ldots , k, \hat{1}, \ldots, \hat{k}\}$ such that $i \in \{1, \ldots, k\}$ and  $\hat{i}\in \{\hat{1}, \ldots, \hat{k}\}$. Suppose $\la = \{i, \hat{j}\}$ with $i\neq j$. In this case we will get an equation of the form
\begin{equation}
\eta_i\alpha_i^{n_1} +\eta_{\hat{j}}\alpha_{\hat{j}}^{n_2}=0.
\end{equation}
Since by our assumption, $\alpha_i$'s are pair wise multiplicatively independent, we get $G(\pi) = \{(0, 0, 0)\}$. Hence, by Lemma \ref{gp} we get an upper bound for the number of these values of $n_1, n_2$ in terms of $k$. Next suppose that $\la = \{i, \hat{i}\}$. Further, if $\al_i^{n_1} + \al_{\hat{i}}^{n_2} \neq 0$, then by Lemma \ref{lem3.4}, we get an upper bound for the number of these values of $n_1, n_2$ in terms of $k$. On the other hand if $\al_i^{n_1} + \al_{\hat{i}}^{n_2} = 0$, then $\al_i$ will be a root of unity, which is a contradiction.

Suppose next that there is a subset $\la$ of $\pi$ such that $\la = \{k+\hat{k}+1, k+\hat{k}+2\}$. Since $\be/\ga$ is not a root of unity, we see that $G(\pi) = \{(0, 0, 0)\}$. So Lemma \ref{gp} yields an upper bound for the number of such values of $n_1, n_2$ in terms of $k$.

Thus, we are left with the case where $\pi$ consists of precisely two sets, say $\la_1$ and $\la_2$ with $k +\hat{k}+1 \in \la_1, k +\hat{k}+2\in \la_2$. Obviously, $|\la_1|, |\la_2| \geq 2$.
Assume that one of these sets, say $\la_1$, has more than two elements. Then there exists $i, j$ with $1 \leq i < j \leq k$ with $i, j \in \la_1$. Since $\al_i/\al_j$ is not a root of unity, we get that $G(\pi) = \{(0, 0, 0)\}$. Thus, Lemma \ref{gp} provides an upper bound in terms of $k$ for the number of these values of $n_1$ once again. That is, we may assume that $|\la_1| = |\la_2| = 2$ and hence, we consider the following cases:
\begin{align*}
&B\be^m=\eta_1\al_1^{n_1}\quad \mbox{and} \quad C\ga^m=\eta_2\al_2^{n_1}\\
&B\be^m=\eta_1\al_1^{n_1}\quad \mbox{and} \quad C\ga^m=\eta_{\hat 2}\al_{\hat{2}}^{n_2}\\
&B\be^m=\eta_{\hat 1}\al_{\hat{1}}^{n_2}\quad \mbox{and} \quad C\ga^m=\eta_{\hat 2}\al_{\hat{2}}^{n_2}\\
&B\be^m=\eta_{\hat 1}\al_{\hat{1}}^{n_2}\quad \mbox{and} \quad C\ga^m=\eta_2\al_2^{n_1}.
\end{align*}
Consider the first case 
\begin{equation}\label{eq44}
B\be^m=\eta_1\al_1^{n_1} \quad \mbox{and} \quad C\ga^m=\eta_2\al_2^{n_1}.
\end{equation} 
In this case, we see that the characteristic polynomial of $(V_n)$ has precisely two distinct roots. In particular, $\al_1, \al_2$ are either rational, or conjugated quadratic algebraic numbers. If in \eqref{eq44} we have $G(\pi) = \{(0, 0, 0)\}$, then we can bound the number of solutions in the usual way. So we may assume that 
$G(\pi) \neq \{(0, 0, 0)\}$. Thus, there exists $r, s\in \Z$ for which
\begin{equation}\label{eq45}
\be^r=\al_1^s \quad \mbox{and} \quad \ga^r=\al_2^s.
\end{equation} 
If $\alpha_1, \alpha_2\in \Q$, then taking conjugates in $K= \Q(\beta) =\Q(\gamma)$, we get that $\alpha_1=\alpha_2$, which is a contradiction. Thus,  $\alpha_1, \alpha_2$ are conjugated quadratic integers. This implies that $\alpha_1, \alpha_2\in K$. Since $\beta$ and $\gamma$ are units in $\Z[\sqrt{d}]$, then $\alpha_1, \alpha_2$ are units in the ring of integers $\mathcal{O}_{K}$. Then multiplying the left and right hand side of \eqref{eq44}, we get
\begin{equation}
BC = (-1)^{n_1}\eta_1\eta_2.
\end{equation}
Hence, the characteristic polynomial of $(V_n)$ is
\[T(x) := (x - \al_1)(x - \al_2) \in \Z[x].\]
However, as the roots of $T(x)$ are units of $\mathcal{O}_{K}$, its constant term is $\pm 1$ and the square-free part of its discriminant equals $d$, which contradicts our assumption.

Now, consider the second case 
\begin{equation}\label{eq46}
B\be^m=\eta_1\al_1^{n_1} \quad \mbox{and} \quad C\ga^m=\eta_{\hat 2}\al_{\hat{2}}^{n_2}.
\end{equation} 
Similarly, we may assume that $G(\pi) \neq \{(0, 0, 0)\}$. Thus, there exists $r, s_1, s_2\in \Z$ for which
\begin{equation}\label{eq47}
\be^r=\al_1^{s_1} \quad \mbox{and} \quad \ga^r=\al_{\hat 2}^{s_2}.
\end{equation} 
Again if $\alpha_1, \alpha_{\hat 2}\in \Q$, then taking conjugates in $K= \Q(\beta) =\Q(\gamma)$, we get that $\alpha_1^{s_1}=\alpha_{\hat 2}^{s_2}$, which is a contradiction as $\alpha_1$ and $\alpha_{\hat 2}$ are multiplicatively independent. Thus,  $\alpha_1, \alpha_{\hat 2}$ are conjugated quadratic integers. Similarly proceeding as above, we will get a contradiction. The other two cases we can proceed similarly and this completes the proof of the theorem.

{\bf Data Availability Statements} Data sharing not applicable to this article as no datasets were generated or analyzed during the current study.

{\bf Acknowledgment:} The first author's work is supported by CSIR fellowship (File no: 09/983(0036)/2019-EMR-I). R.P. and S.S.R. supported by a grant from National Board for Higher Mathematics (NBHM), Sanction Order No: 14053. Furthermore, S.S.R. is partially supported by grant from Science and Engineering Research Board (SERB)(File No.:CRG/2022/000268).  This paper started during first author's visit to the Department of Mathematics, NIT Calicut, India. 

{\bf Declarations}

{\bf Conflict of interest} On behalf of all authors, the corresponding author states that there is no Conflict of interest.

\end{document}